\documentclass[12pt]{article}

\usepackage{amsfonts}
\usepackage{graphics}
\usepackage{amssymb}

\newtheorem{theorem}{Theorem}[section]
\newtheorem{lemma}{Lemma}[section]
\newtheorem{definition}{Definition}[section]

\newtheorem{remark}{Remark}[section]

\begin{document}

\begin{center}
{\Large   $L^{p}-$solutions of the stochastic transport equation }

\end{center}

\vspace{0.3cm}

\begin{center}

{\large Pedro Catuogno\footnote{Research partially supported by CNPQ 302704/2008-6.} and
Christian Olivera.}\\

\textit{Departamento de
 Matem\'{a}tica, Universidade Estadual de Campinas, \\ F. 54(19) 3521-5921 � Fax 54(19)
3521-6094\\ 13.081-970 -
 Campinas - SP, Brazil. e-mail:  colivera@ime.unicamp.br}
\end{center}

\vspace{0.3cm}
\begin{center}
\begin{abstract}
We consider the stochastic transport linear equation and we prove existence and uniqueness
of weak  $L^{p}-$solutions.   Moreover, we obtain a representation of the general solution
and a  Wong-Zakai principle for this equation. We make only minimal assumptions, similar to the deterministic problem.  The proof is supported on
 the generalized It\^o-Ventzel-Kunita formula (see \cite{Ku2}) and the theory of Lions-DiPerna on transport linear equation (see \cite{DL}).
\end{abstract}
\end{center}


\vspace{0.3cm} \noindent {\bf MSC2000 subject classification:} 60H10
, 60H15  .

\section {Introduction}

In this article we establish global existence and uniqueness of
solution of the transport linear equation with a stochastic perturbation. Namely, we consider the following equation:

\begin{equation}\label{trasport}
 \left \{
\begin{array}{lll}
    \frac{d}{dt}u(t, x) + b(t, x) \nabla u(t, x) + \nabla u(t,x) \frac{d B_{t}}{dt}=0,\\
 u(0, x) = u_{0}(x)\in L^{p}(\mathbb{R}^{d}),
\end{array}
\right .
\end{equation}

\noindent where $B_{t} = (B_{t}^{1},...,B _{t}^{d} )$ is a
standard Brownian motion in $\mathbb{R}^{d}$ and  the stochastic
integration is taken in the Stratonovich sense.

This equation has been treated  for the case
 $u_{0}(x)\in L^{\infty}(\mathbb{R}^{d})$ (see \cite{FGP2} and \cite{Ku})
via the stochastic characteristic method. Our aim here is to prove the
existence, uniqueness and regularity when the initial data $u_{0}(x)\in L^p(\mathbb{R}^{d})$ for $p \in [1,\infty)$.
Some partial results are presented in \cite{BrisLion}, where the case $u_{0}(x)\in L^1(\mathbb{R}^{d}) \cap L^{\infty}(\mathbb{R}^{d})$ was studied.

The theory of renormalized solutions of the linear transport equation was introduced by DiPerna and Lions in a celebrated paper \cite{DL}.
They deduced the existence, uniqueness and stability results for ordinary differential equations with rough coefficients from corresponding results on the
associated linear transport equation. Similar results were obtained in \cite{CruCi} by
taking the standard Gaussian measure as the reference measure. Ambrosio \cite{ambrisio} generalized the
results to the case where the coefficients have only bounded variation regularity by considering the continuity equation.
These results have recently been generalized into different settings, \cite{AF} and \cite{FangLuo} for infinite dimensional spaces,
\cite{Figa}  and \cite{Zanng} for generalizations to transport-diffusion
equations and its associated stochastic differential equations.

 We prove existence and uniqueness of  weak  $L^{p}-$solution using
the generalized It\^o-Ventzel-Kunita formula (see Theorem 8.3 of  \cite{Ku2})
and the results on existence and uniqueness for the deterministic  transport linear equation
(see for example \cite{DL} and \cite{BrisLion}). We give a Wong-Zakai principle for the
stochastic transport equation (\ref{trasport}), this principle is
proved via stability properties of the deterministic transport linear equation. We would like to mention that our approach clearly
differs from that one in \cite{FGP2}, however this article has been a source of inspiration for us.

The plan of exposition is as follows: In section 2  we prove existence of
 weak  $L^{p}-$solutions and we point some extensions.
In section 3, we show a uniqueness theorem for weak  $L^{p}-$solutions. Finally, in section 4,
we establish a Wong-Zakai principle for the SPDE (\ref{trasport}).

Through of this paper we fix a stochastic basis with a
$d$-dimensional Brownian motion $(\Omega, \mathcal{F}, \{
\mathcal{F}_t: t \in [0,T] \}, \mathbb{P}, (B_{t}))$.

\section{Stochastic transport equation. Existence of weak solutions}

\begin{definition}\label{defisolu}
 A   weak  $L^{p}-$solution of the  Cauchy problem (\ref{trasport}) is a stochastic process
    $u\in L�^{\infty}( �\Omega\times[0, T],�L^{p}(\mathbb{R}^{d}))$
such that, for every test function�
  $\varphi \in C_{0}^{\infty}(\mathbb{R}^{d})$, the process $\int u(t,
  x)\varphi(x)
  dx$ has a continuous modification which is a
$\mathcal{F}_{t}$-semimartingale and satisfies

\[
\int u(t,x) \varphi(x) dx= \int u_{0}(x) \varphi(x) \ dx
\]
\[
+\int_{0}^{t} \int b(s,x) \nabla \varphi(x) u(s,x) \ dx ds
+ \int_{0}^{t} \int div \ b(s,x) \varphi(x) u(s,x) \ dx ds \
\]
\[
+\sum_{i=0}^{d} \int_{0}^{t} \int  D_{i}\varphi(x) u(s,x) \ dx \circ
dB_{s}^{i}
\]
\end{definition}
\noindent We shall always assume that

\begin{equation}\label{con1}
b\in L^{1}([0,T], (L_{loc}^{1}(\mathbb{R}^{d}))^d)
\end{equation}

\noindent We observe that this definition makes sense if we assume

\begin{equation}\label{con2}
b\in L^{1}([0,T], (L_{loc}^{q}(\mathbb{R}^{d}))^d)
\end{equation}

\noindent where $q$ is the conjugate exponent of p.

\begin{lemma}\label{lemaexis} Let $p\in[1,\infty),  u_{0}\in
L^{p}(\mathbb{R}^{d})$. Assume (\ref{con1}), (\ref{con2}) and that

\begin{equation}\label{cond3}
div \ b\in L^{1}([0,T], L^{\infty}(\mathbb{R}^{d}))
\end{equation}

\noindent  Then there exits a weak $L^{p}-$solution $u$ of the SPDE
(\ref{trasport}).
\end{lemma}

\begin{proof} {\large Step 1} (auxiliary transport equation) We
considerer  the following auxiliary transport equation

\begin{equation}\label{Auxilia1}
 \left \{
\begin{array}{lll}
v_t + b(t,x+B_{t})  \nabla v(t,x)& = & 0  \\
v(0,x) & = & u_{0}(x), \ x \in\mathbb{R}^{d}.
\end{array}
\right.
\end{equation}

\noindent According to an easy modification of \cite{DL},
Proposition II.1 (taking only test functions defined on $\mathbb{R}^d$) there is a solution $v \in L^{\infty}([0, T]
\times \Omega,L^{p}(\mathbb{R}^{d}))$  of the  equation
(\ref{Auxilia1}) in the sense that it satisfies

\[
\int v(t,x) \varphi(x) dx= \int u_{0}(x) \varphi(x) \ dx
\]
\begin{equation}\label{Auxilia2}
+ \int_{0}^{t} \int b(s,x+B_{s}) \nabla \varphi(x) v(s,x) \ dx ds +
\int_{0}^{t} \int div \ b(s,x+ B_{s}) \varphi(x) v(s,x) \ dx ds \
\end{equation}

\noindent {\large Step 2} (  Solution via It\^o-Ventzel-Kunita
formula)

\noindent Applying the It\^o-Ventzel-Kunita formula to $F(y)=\int
u(t,x) \varphi(x+y) dx$ (see Theorem 8.3 of \cite{Ku2}) we
obtain that

\[
\int v(t,x) \varphi(x+B_{t}) dx
\]
\noindent is equal to

\[
\int u_{0}(x) \varphi(x) \ dx + \int_{0}^{t} \int b(s,x+B_{s})
\nabla \varphi(x+B_{s}) v(s,x) \ dx ds
\]
\[
+\int_{0}^{t} \int div \ b(s,x+B_{s}) \varphi(x+B_{s}) v(s,x) \ dx
ds \
\]
\[
  + \sum_{i=1}^{d}  \int_{0}^{t} \int
 v(s,x) \frac{\partial}{\partial
y_{i}}\varphi(x+B_{s}) dx \circ dB_{s}^{i}.
\]

\noindent  We note that  $\frac{\partial}{\partial y_{i}}
\varphi(x+B_{s})=\frac{\partial}{\partial x_{i}}
\varphi(x+B_{s})$.  Thus

\[
\int v(t,x) \varphi(x+B_{t}) dx= \int u_{0}(x) \varphi(x) \ dx
\]
\[
+ \int_{0}^{t} \int b(s,x+B_{s}) \nabla \varphi(x+B_{s}) v(s,x) \ dx
ds + \int_{0}^{t} \int div \ b(s,x+B_{s}) \varphi(x+B_{s}) v(s,x) \
dx ds \
\]
\begin{equation}\label{auxilia3}
   + \sum_{i=1}^{d}  \int_{0}^{t} \int
 v(s,x) D_{i}\varphi(x+B_{s}) dx \circ dB_{s}^{i}.
\end{equation}

\noindent From  the equation (\ref{auxilia3}) we follow that
$u(t,x):=v(t,x-B_{t})$ is a weak  $L^{p}-$solution
  of the SPDE (\ref{trasport}).
\end{proof}

\begin{remark}\label{coment2} We observe that the same proof is valid if
 we assume that $b(t,x,\cdot)\in \mathcal{F}_t$ for all
$(t,x) \in [0,T]\times\mathbb{R}^{d}$ and verifies
 (\ref{con2}) and (\ref{cond3}) almost surely for $\omega \in \Omega$. This results gives a
partial  answer  about the   existence of  solution for the SPDE
(\ref{trasport}) with
 stochastic coefficient (see Introduction of \cite{FGP2})
\end{remark}

\begin{remark}\label{coment2} We would to note that the same proof works for the equation

\begin{equation}\label{trasportmod}
 \left \{
\begin{array}{lll}
    \frac{d}{dt}u(t, x) + b(t, x) \nabla u(t, x) + \nabla u(t,x) \frac{d B_{t}}{dt} + c(t,x)u=f(t,x),\\
 u(0, x) = u_{0}(x)\in L^{p}(\mathbb{R}^{d}),
\end{array}
\right .
\end{equation}

 \noindent where $c$ and $f$ satisfy the conditions  of the Proposition II.1 of \cite{DL}
\end{remark}

\begin{remark} We mention some  future works
\begin{enumerate}


\item[a)] The case  that $b(t,\cdot, \omega)$  is a nonadapted
process could be studied via a Generalized
It\^o formula for nonadapted process (see by example \cite{OconePardoux}).

\item[b)] The stochastic transport equations with other noises could be studied via the stochastic calculus via regularization (see
\cite{CoveiloRuso} and \cite{{FLRU}}).

\item[c)] For initial data and coefficients more singular, a
possible approach is to study the transport equation in the sense
of generalized function algebras, see for instance
\cite{albe} and \cite{Russo}. For a new approach see
\cite{CO1} and \cite{CO2}.

\end{enumerate}

\end{remark}

\section{Uniqueness}

\noindent In this section, we shall present a uniqueness theorem
for the SPDE (\ref{trasport}) under similar  conditions to the
deterministic case (see for instance \cite{DL}
and \cite{BrisLion}).

\begin{theorem}\label{uni} Let $p \in [1,\infty)$. Assume that $  div \ b\in L^{1}([0,T], L^{\infty}(\mathbb{R}^{d}))$,
   $ b \in L^{1}([0,T], (W_{loc}^{1,q}(\mathbb{R}^{d}))^d)$ and $\frac{|b|}{1+|x|}\in L^{1}([0,T],L^{1}(\mathbb{R}^{d}))+L^{1}([0,T],L^{\infty}(\mathbb{R}^{d}))$ . Then, for every   $u_{0}\in L^{p}(\mathbb{R}^{d})$ there
exists a unique  weak $L^{p}-$solution of the Cauchy problem
(\ref{trasport}).
\end{theorem}

\begin{proof} By linearity we have to show that a weak
$L^{p}-$solution with initial condition $u_{0}(x)=0$ vanishes
identically.  Applying the It\^o-Ventzel-Kunita formula (see Theorem
8.3 of \cite{Ku2} ) to $F(y)=\int u(t,x) \varphi(x-y) \ dx $, we obtain that

\[
\int u(t,x) \varphi(x-B_{t}) dx
\]

\noindent is equal to

\[
\int_{0}^{t} \int b(s,x) \nabla \varphi(x-B_{s}) u(s,x) \ dx
ds + \int_{0}^{t} \int div \ b(s,x) \varphi(x-B_{s}) u(s,x) \
dx ds \
\]
\[
 + \sum_{i=1}^{d}  \int_{0}^{t} \int
 u(s,x) D_{i}\varphi(x-B_{s}) dx \circ dB_{s}^{i}.  + \sum_{i=1}^{d}  \int_{0}^{t} \int
 u(s,x) \frac{\partial}{\partial
y_{i}}[ \varphi(x-B_{s})] dx \circ dB_{s}^{i}.
\]

\noindent We observe that $\frac{\partial}{\partial y_{i}}[
\varphi(x-B_{s})]=-\frac{\partial}{\partial
x_{i}}\varphi(x-B_{s})$. Thus $V(t,x)=u(t,x+ {B_{t}})$ verifies

\[
\int V(t,x) \varphi(x) dx = \int_{0}^{t} \int b(s,x+B_{s}) \nabla
\varphi(x) V(s,x)dx ds
\]

\[
+ \int_{0}^{t} \int div \ b(s,x+B_{s}) \varphi(x) V(s,x) \ dx ds .
\]

\noindent Let $\phi_{\varepsilon}$ be a standard mollifier. Since
$b(s,x+B_{s})$ satisfies $\mathbb{P}$ a.s the hypothesis
of our Theorem, then by the Commuting Lemma (see Lemma II.1 of
\cite{DL}), $V_{\varepsilon}(t,x)=V(t,.)\ast \phi_{\varepsilon}$
verifies
\[
\lim_{\varepsilon \rightarrow
0}\frac{dV_{\varepsilon}}{dt}+b(t,x+B_{t})\nabla V_{\varepsilon}=0
~ ~~\mathbb{P} \ a.s \mbox{ in }\ L^{1}([0,T],
L^{1}_{loc}(\mathbb{R}^{d})).
\]

\noindent We deduce that if $\beta\in
C^{1}(\mathbb{R})$ and $\beta^{\prime}$ is bounded, then

\begin{equation}\label{norma}
\frac{d\beta(V)}{dt}+b(t,x+B_{t}) \nabla \beta(V)=0.
\end{equation}

\noindent Now, following the same steps in the proof of Theorem II.
2 of \cite{DL}, we define for each $M \in (0,\infty)$ the function
$\beta_M(t)=(|t| \wedge M )^p$ and obtain that
\[
\frac{d}{dt}\int \beta_M(V(t,x))dx \leq C \int \beta_M(V(t,x))dx.
\]
Taking expectation we have that
\[
\frac{d}{dt}\int \mathbb{E}(\beta_M(V(t,x)))dx \leq C \int
\mathbb{E}(\beta_M(V(t,x)))dx.
\]
From Gronwall Lemma we conclude that $\beta_M(V(t,x))=0$. Thus
 $u=0$.
\end{proof}

\begin{remark}
We observe that the unique solution $u(t,x)$ has the
representation $u(t,x)=v(t, x-B(t))$, where $v$ satisfies
(\ref{Auxilia2}). From Corollary 2.2 of \cite{DL}, we follow that
$v$ belongs to $C([0,T], L^p(\mathbb{R}^d))$. Thus $u \in C([0,T],
L^p(\mathbb{R}^d))$.
\end{remark}

\section{ Wong-Zakai principle}

\noindent The Wong-Zakai principle says that the solutions to
equations where the noise is approximated by more regular
processes converge to the solution of the stochastic differential
equation with Stratonovich integrals. We mention that there exist
several works about of the  Wong-Zakai principle for SPDE (see for instance \cite{BRFL}, \cite{GYS}  and
references). Our method for prove this principle is based on the
stability properties for the renormalized solutions of the
deterministic transport equation.

\noindent Now, we considerer approximations of the Brownian
motion, by continuous and bounded variation processes $B_n (t)$
such that
\[
\lim_{n \rightarrow \infty}B_n(t) = \ B(t) \ \ \mathbb{P}-a.s. \
\mbox{ uniformly in }  t.
 \]
We considerer the next equations

\begin{equation}\label{trasportap}
 \left \{
\begin{array}{lll}
    \frac{d}{dt}u_{n}(t, x) + b(t, x) \nabla u_{n}(t, x) + \nabla u_{n}(t,x) \frac{d B_{n}}{dt}(t)=0,\\
 u_{n}(0, x) = u_{0}(x)\in L^{p}(\mathbb{R}^{d}).
\end{array}
\right .
\end{equation}

\begin{definition}\label{defisoluap}
 A    weak  $L^{p}-$solution of the Cauchy problem (\ref{trasportap}) is a stochastic process
    $u_{n}\in L�^{\infty}( �\Omega\times[0, T],�L^{p}(\mathbb{R}^{d}))$
such that, for every test function
  $\varphi \in C_{0}^{\infty}(\mathbb{R}^{d})$, the process $\int u_{n}(t,
  x)\varphi(x)
  dx$ has a continuous modification  and satisfies

\[
\int u_{n}(t,x) \varphi(x) dx= \int u_{0}(x) \varphi(x) dx
\]
\[
+ \int_{0}^{t} \int b(s,x) \nabla \varphi(x) u_{n}(s,x) dx ds +
\int_{0}^{t} \int div \ b(s,x) \varphi(x) u_{n}(s,x) dx ds \
\]
\[
+\sum_{i=0}^{d} \int_{0}^{t} \int  D_{i}\varphi(x) u_{n}(s,x) dx
dB_n(s).
\]
\end{definition}

\begin{lemma}\label{lemaexis} Let $u_{0}\in
L^{p}(\mathbb{R}^{d})$ where  $p\in[1,\infty)$. Suppose that $div
\ b\in L^{1}([0,T], L^{\infty}(\mathbb{R}^{d}))$ and (\ref{con1})
and (\ref{con2}) holds. Then there exits a weak $L^{p}-$solution
$u_{n}$ of the SPDE (\ref{trasportap}).
\end{lemma}

\begin{proof}  We considerer  the following auxiliary transport equation

\begin{equation}\label{Auxilia1ap}
 \left \{
\begin{array}{lll}
\frac{dv_{n}}{dt} & = &  b(t,x+B_{n}(t))  \nabla v_{n}(t,x)  \\
v_n(0,x) & = & u_{0}(x).
\end{array}
\right.
\end{equation}

\noindent According to a small modification of \cite{DL}, Proposition II.1
 there is  a solution $v_{n}(t,x)$  of the  equation (\ref{Auxilia1ap}) in the sense that

\[
\int v_{n}(t,x) \varphi(x) dx
\]
is equal to
\begin{equation}\label{vap}
\int u_{0}(x) \varphi(x) dx+ \int_{0}^{t} \int b(s,x+ B_n(s))
\nabla \varphi(x) v_{n}(s,x) dx ds
\end{equation}
\[
+ \int_{0}^{t} \int div \ b(s,x+ B_n(s)) \varphi(x) v_{n}(s,x) dx
ds
\]
 \noindent Following the proof of the Lemma \ref{lemaexis} we get that

\begin{equation}
u_{n}(t,x)=v_{n}(t,x-B_n(t))
\end{equation}

\noindent is a weak  $L^{p}(\mathbb{R}^{d})-$ solution of the SPDE
(\ref{trasportap}).

\end{proof}

The uniqueness of the approximate problem (\ref{trasportap}),
follows changing $B$ by $B_n$ in the proof of the Theorem
\ref{uni}.

\begin{theorem} Let $p \in [1,\infty)$ and $\frac{1}{p}+\frac{1}{q}=1$. Assume that $  div \ b\in
L^{1}([0,T], L^{\infty}(\mathbb{R}^{d}))$, $ b \in L^{1}([0,T],
(W_{loc}^{1,q}(\mathbb{R}^{d}))^d)$ and $\frac{|b|}{1+|x|}\in
L^{1}([0,T],L^{1}(\mathbb{R}^{d}))+L^{1}([0,T],L^{\infty}(\mathbb{R}^{d}))$.
Then, for every  $u_{0}\in L^{p}(\mathbb{R}^{d})$ there exists a
unique  weak $L^{p}-$solution of the Cauchy problem
(\ref{trasportap})
\end{theorem}

\noindent Finally, we prove our Wong-Zakai principle.

\begin{theorem} Let $p \in [1,\infty)$ and $\frac{1}{p}+\frac{1}{q}=1$. Assume that $  div \ b\in L^{1}([0,T], L^{\infty}(\mathbb{R}^{d}))$,
 $ b \in L^{1}([0,T], (W_{loc}^{1,q}(\mathbb{R}^{d}))^d)$ and $\frac{|b|}{1+|x|}\in
L^{1}([0,T],L^{1}(\mathbb{R}^{d}))+L^{1}([0,T],L^{\infty}(\mathbb{R}^{d}))$.  Let $u$ and $u_{n}$ are weak
  $L^{p}-$solutions of the SPDE (\ref{trasport}) and (\ref{trasportap}) respectively. Then $u_{n}$
converges to $u$, \ $\mathbb{P}$ \ a.s \ in   $C( [0,
T],L^{p}(\mathbb{R}^{d}))$.
\end{theorem}

\begin{proof} We know that $u(t, x)=v(t,x-B(t))$ and $u_{n}(t, x)=v_{n}(t,x-B_n(t))$
where $v(t,x)$ and $v_n(t,x)$ satisfies (\ref{Auxilia2}) and
(\ref{vap}) respectively. By Theorem 2.4 of \cite{DL} we have that

\[
\lim_{n \rightarrow \infty}v_{n}(t,x) = v(t,x) \ \ \mathbb{P} \
a.s \mbox{ in } C( [0, T],L^{p}(\mathbb{R}^{d})).
\]
\noindent From this fact we obtain immediately that

\[
\lim_{n \rightarrow \infty}u_{n}(t,x) = u(t,x) \ \ \mathbb{P} \
a.s \mbox{ in } C( [0, T],L^{p}(\mathbb{R}^{d})).
\]

\end{proof}

\noindent {\it Acknowledgements}
C.Olivera  would like to thank the hospitality of the Universidad Estadual de Campinas(UNICAMP), where
this work was initiated, and of Instituto Nacional de Matem´atica Pura e Aplicada(IMPA), where this work was finished.
C. Olivera thanks to  FAPESP  for the financial support  during your stay in UNICAMP.

\end{document}